\documentclass[]{article}
\usepackage{amsthm}
\usepackage{amsfonts}
\usepackage{amsmath,url}

\newtheorem{theorem}{Theorem}[section]
\newtheorem{corollary}[theorem]{Corollary}

\newtheorem{example}[theorem]{Example}

\title{Heun functions and combinatorial identities}
\author{Adina  B\u{a}rar\thanks{Technical University of Cluj-Napoca, Department of Mathematics, Memorandumului Street 28, 400114, Cluj-Napoca, Romania}, Gabriela Raluca Mocanu\thanks{Romanian Academy, Cluj-Napoca Branch, Astronomical Institute, Cire\c{s}ilor Street 19, 400487, Cluj-Napoca, Romania,
gabriela.mocanu$@$academia-cj.ro}, Ioan Ra\c{s}a\thanks{Technical University of Cluj-Napoca, Department of Mathematics \newline, Memorandumului Street 28, 400114, Cluj-Napoca, Romania}}

\date{}
\begin{document}

\maketitle

Subject Class: 33E30, 94A17, 05A19.
\newline \newline
\textbf{Abstract}
\newline
We give closed forms for several families of Heun functions related to classical entropies. By comparing two expressions of the same Heun function, we get several combinatorial identities generalizing some classical ones.

\vspace{1.5cc}
\begin{center}
{\bf 1. INTRODUCTION}
\end{center}
Consider the general Heun equation (see, e.g., \cite{12}, \cite{4}, \cite{2} and the references therein)
\begin{equation}
u''(x)+\left ( \frac{\gamma}{x}+\frac{\delta}{x-1}+\frac{\epsilon}{x-a} \right )u'(x)+\frac{\alpha \beta x-q}{x(x-1)(x-a)}u(x)=0,\label{eq:1.1}
\end{equation}
where $a\notin \{0,1\}$, $\gamma \notin \{ 0,-1,-2,\dots\}$ and $\alpha+\beta+1=\gamma +\delta+\epsilon$. Its solution $u(x)$ normalized by the condition $u(0)=1$ is called the \emph{(local) Heun function} and is denoted by $Hl(a,q;\alpha,\beta;\gamma, \delta;x)$.

The confluent Heun equation is
\begin{equation}
u''(x)+\left ( 4p+\frac{\gamma}{x}+\frac{\delta}{x-1} \right ) u'(x)+\frac{4p\alpha x-\sigma}{x(x-1)}u(x)=0, \label{eq:1.2}
\end{equation}
where $p\neq 0$. The solution $u(x)$ normalized by $u(0)=1$ is called the \emph{confluent Heun function} and is denoted by $HC(p,\gamma,\delta,\alpha,\sigma;x)$.

It was proved in \cite{5} that
\begin{equation}
Hl \left(\frac{1}{2},-n;-2n,1;1,1;x \right) = \sum _{k=0}^n \left ( {n\choose k} x^k (1-x)^{n-k} \right )^2,\label{eq:1.3}
\end{equation}
\begin{equation}
Hl \left(\frac{1}{2},n;2n,1;1,1;-x \right) = \sum _{k=0}^\infty \left ( {n+k-1\choose k} x^k (1+x)^{-n-k} \right )^2,\label{eq:1.4}
\end{equation}\
\begin{equation}
HC \left(n,1,0,\frac{1}{2},2n;x \right) = \sum _{k=0}^\infty \left ( e^{-nx} \frac{(nx)^k}{k!} \right )^2.\label{eq:1.5}
\end{equation}

More general results, providing closed forms of the functions\\ $Hl \left(\frac{1}{2},-2n\theta;-2n,2\theta;\gamma,\gamma;x \right)$ and $Hl \left(\frac{1}{2},2n\theta;2n,2\theta;\gamma,\gamma;x \right)$, and explicit expressions for some confluent Heun functions can be found in \cite{6}.

In this paper we give closed forms for several families of Heun functions and confluent Heun functions, extending \eqref{eq:1.3}, \eqref{eq:1.4} and \eqref{eq:1.5}. Basic tools will be the results of \cite{3} and \cite{10} concerning the derivatives of Heun functions, respectively confluent Heun functions; see also \cite{1} and \cite{6}.

By comparing two expressions of the same Heun function, we get several combinatorial identities; very particular forms of them can be traced in the classical book \cite{9}.

Let us mention that the families of (confluent) Heun functions mentioned above are naturally related to some classical entropies: see \cite{1}, \cite{5}, \cite{6}, \cite{7}, \cite{8}.

Throughout the paper we shall use the notation
\begin{equation*}
(x)_0:=1, \quad (x)_k:=x(x+1)\dots (x+k-1), \quad k\geq 1,
\end{equation*}
\begin{equation}
a_{nj}:=4^{-n} {2j \choose j} {2n-2j \choose n-j}, \label{eq:1.6}
\end{equation}
\begin{equation}
r_{nj}:= {n \choose j}^{-1}a_{nj}. \label{eq:1.7}
\end{equation}
\vspace{2cc}

\vspace{1.5cc}
\begin{center}
{\bf 2. HEUN FUNCTIONS}
\end{center}
\subsection*{Let $\alpha \beta \neq 0$.}

As a consequence of the results of \cite{3} we have (see \cite[Prop. 1]{6} and \cite[(14)]{6}):
\begin{eqnarray}
Hl \left ( \frac{1}{2}, \frac{1}{2}(\alpha +2)(\beta +2);\alpha +2, \beta +2; \gamma +1 , \gamma +1; x \right ) = \label{eq:2.1} \\
\frac{\gamma}{\alpha \beta} (1-2x)^{-1}\frac{d}{dx}Hl \left ( \frac{1}{2}, \frac{1}{2}\alpha \beta; \alpha , \beta ; \gamma , \gamma ; x \right ). \nonumber
\end{eqnarray}

From \cite[(6)]{6}, \cite[(22)]{6} and \eqref{eq:1.6} we obtain
\begin{equation}
Hl\left ( \frac{1}{2}, -n; -2n,1 ; 1,1 ; x \right ) = \sum _{j=0}^n a_{nj} (1-2x)^{2j}.\label{eq:2.2}
\end{equation}

\begin{theorem}\label{thm:2.1}
Let $0\leq m\leq n$. Then
\begin{eqnarray}
&&Hl \left ( \frac{1}{2}, (2m+1)(m-n);2(m-n),2m+1; m +1 , m +1; x \right ) \label{eq:2.3} \\
&=&4^m {n \choose m}^{-1} {2m \choose m}^{-1} \sum _{j=0}^{n-m} {m+j \choose m}a_{n,m+j}(1-2x)^{2j} \nonumber\\
&=& \sum _{j=0}^{n-m} 4^j {n-m \choose j} \frac{(m+1/2)_j}{(m+1)_j} (x^2-x)^j.\nonumber
\end{eqnarray}
\end{theorem}
\begin{proof}
We shall prove the first equality by induction with respect to $m$.

For $m=0$, it follows from \eqref{eq:2.2}. Suppose that it is valid for a certain $m<n$. Then \eqref{eq:2.1} implies
\begin{equation*}
Hl \left ( \frac{1}{2},\frac{1}{2} (2m+3)(m+1-n);2(m+1-n),2m+3;m+2,m+2;x\right )
\end{equation*}
\begin{equation*}
=\frac{(m+1)(1-2x)^{-1}}{2(m-n)(2m+1)}\frac{d}{dx}Hl \left ( \frac{1}{2},(2m+1)(m-n);2(m-n),2m+1;m+1,m+1;x\right )
\end{equation*}
\begin{equation*}
=\frac{(m+1)(1-2x)^{-1}}{2(m-n)(2m+1)}4^m {n \choose m}^{-1}{2m \choose m}^{-1}\sum_{i=1}^{n-m} {m+i \choose m} a_{n,m+i} (-4i)(1-2x)^{2i-1}
\end{equation*}
\begin{equation*}
=4^{m+1} {n \choose m+1}^{-1}{2m+2 \choose m+1}^{-1}\sum_{j=0}^{n-m-1} {m+1+j \choose m+1} a_{n,m+1+j} (1-2x)^{2j},
\end{equation*}
and so the desired equality is true for $m+1$; this finishes the proof by induction.

In order to prove that the first member and the last member of~\eqref{eq:2.3} are equal, it suffices to use~\cite[Th. 1]{6} with $\gamma = m+1$, $\theta = m+\frac{1}{2}$, and $n$ replaced by $n-m$.
\end{proof}

\begin{corollary}\label{cor:2.2}
Let $0\leq i\leq n-m$, $0\leq j\leq n-m$. Then
\begin{eqnarray}
&&\sum_{j=i}^{n-m}(-1)^{j-i}{n-m \choose j} \frac{(m+1/2)_j}{(m+1)_j}{j \choose i}=\nonumber \\ &&=4^m {n \choose m}^{-1}{2m \choose m}^{-1}{m+i \choose m}a_{n,m+i},\label{eq:2.4}
\end{eqnarray}
\begin{eqnarray}
\sum_{i=j}^{n-m}{m+i \choose m} {i \choose j} a_{n,m+i} = 4^{-m} {n \choose m}{2m \choose m} \frac{(m+1/2)_j}{(m+1)_j}{n-m \choose j}.\label{eq:2.5}
\end{eqnarray}
\end{corollary}

\begin{proof}
It suffices to combine the last equality in~\eqref{eq:2.3} with
\begin{equation*}
(x^2-x)^j = 4^{-j}\left ( (1-2x)^2-1 \right )^j,
\end{equation*}
respectively with
\begin{equation*}
(1-2x)^{2j} = \left ( 1+4(x^2-x) \right )^j.
\end{equation*}
\end{proof}

\begin{example}
\end{example}
For $i=m=0$, \eqref{eq:2.4} reduces to
\begin{equation}
\sum _{j=0}^n \left ( -\frac{1}{4}\right)^j {n\choose j}{2j \choose j}=4^{-n}{2n \choose n},\label{eq:2.6}
\end{equation}
which is (3.85) in \cite{9}.

For $j=m=0$, \eqref{eq:2.5} becomes
\begin{equation}
\sum _{i=0}^n {2i\choose i}{2n-2i \choose n-i}=4^n,\label{eq:2.7}
\end{equation}
which is (3.90) in \cite{9}.

\subsection*{}
From \cite[(7)]{6}, \cite[(23)]{6} and \eqref{eq:1.6} we know that
\begin{equation}
Hl\left ( \frac{1}{2},n+1;2n+2,1;1,1;x \right ) =\sum_{j=0}^n a_{nj}(1-2x)^{2j-2n-1}. \label{eq:2.8}
\end{equation}

\begin{theorem}\label{thm:2.4}
For $m\geq 0$ we have
\begin{eqnarray}
&&Hl\left ( \frac{1}{2},(2m+1)(m+n+1);2(m+n+1),2m+1;m+1,m+1;x \right ) \nonumber\\
&=&{n+m\choose n}^{-1}\sum_{j=0}^n {2n+2m-2j \choose 2m}{n+m-j \choose m}^{-1}a_{nj}(1-2x)^{2j-2n-2m-1}\nonumber\\
&=&(1-2x)^{-2n-2m-1}\sum_{j=0}^n4^j{n \choose j}\frac{(1/2)_j}{(m+1)_j}(x^2-x)^j.\nonumber
\end{eqnarray}
\end{theorem}
\begin{proof}
As in the proof of Theorem~\ref{thm:2.1}, the first equality can be proved by induction with respect to $m$, if we use \eqref{eq:2.8} and \eqref{eq:2.1}. The equality of the first member and the last member follows from \cite[Cor. 2]{6} by choosing $\gamma=m+1$, $\theta=m+1/2$, and replacing $n$ by $n+m+1$.
\end{proof}

\begin{corollary}\label{cor:2.5}
Let $0\leq i \leq n$, $0\leq j \leq n$. Then
\begin{eqnarray}
&&\sum _{j=i}^n(-1)^{j-i} {n \choose j}\frac{(1/2)_j}{(m+1)_j}{j\choose i}\nonumber \\ &=&{2n+2m-2i \choose 2m}{n+m \choose n}^{-1}{n+m-i \choose m}^{-1}a_{ni},\label{eq:2.9}
\end{eqnarray}
\begin{eqnarray}
&&\sum _{i=j}^n{2n+2m-2i \choose 2m}{n+m-i \choose m}^{-1}{i \choose j}a_{ni} \nonumber \\&=& {n+m \choose n} {n \choose j}\frac{(1/2)_j}{(m+1)_j}.\label{eq:2.10}
\end{eqnarray}
\end{corollary}

The proof is similar to the proof of Corollary~\ref{cor:2.2}.

For $i=m=0$, \eqref{eq:2.9} reduces to \eqref{eq:2.6}, i.e., (3.85) in \cite{9}.

For $j=m=0$, \eqref{eq:2.10} reduces to \eqref{eq:2.7}, i.e., (3.90) in \cite{9}.

\subsection*{}
Let again $\alpha \beta \neq 0$. According to the results of \cite{3} (see \cite[Prop. 1]{6} and \cite[(15)]{6}), we have
\begin{eqnarray}
Hl\left ( \frac{1}{2}, \frac{1}{2} (2\gamma-\alpha)(2\gamma-\beta);2\gamma-\alpha,2\gamma-\beta;\gamma+1,\gamma+1;x \right )\nonumber\\
=\frac{\gamma}{\alpha \beta}(1-2x)^{\alpha+\beta+1-2\gamma}\frac{d}{dx}Hl\left ( \frac{1}{2},\frac{1}{2}\alpha\beta;\alpha,\beta;\gamma,\gamma;x \right ).\label{eq:2.11}
\end{eqnarray}

Using \eqref{eq:2.2}, \eqref{eq:2.11} and the above methods of proof, we obtain the following identities:
\begin{eqnarray}
&&Hl\left ( \frac{1}{2}, (2k+1)(k-n);2(k-n),2k+1;2k+1,2k+1;x \right )\nonumber\\
&=&4^k {n+k \choose n}^{-1}{n \choose k}^{-1}\sum _{i=0}^{n-k}{2n-2i \choose 2k}{n \choose i}r_{n,k+i}(1-2x)^{2i}\nonumber \\
&=&\sum _{j=0}^{n-k}4^j {n-k \choose j}\frac{(k+1/2)_j}{(2k+1)_j}(x^2-x)^j, \quad 0\leq k \leq n.\label{eq:2.12}
\end{eqnarray}

As a consequence of \eqref{eq:2.12} one gets
\begin{eqnarray}
&&\sum_{j=i}^{n-k}(-1)^{j-i} {n-k \choose j} \frac{(k+1/2)_j}{(2k+1)_j}{j \choose i}\label{eq:2.13} \\
&=&4^k {n+k \choose n}^{-1}{n \choose k}^{-1}{2n-2i \choose 2k}{n \choose i}r_{n,k+i}, \quad 0\leq i \leq n-k,\nonumber
\end{eqnarray}
\begin{eqnarray}
&&\sum_{i=j}^{n-k}{2n-2i \choose 2k}{n \choose i}{i\choose j}r_{n,k+i}\label{eq:2.14} \\
&=&4^{-k} {n+k \choose n}{n \choose k}{n-k \choose j} \frac{(k+1/2)_j}{(2k+1)_j}, \quad 0\leq j \leq n-k.\nonumber
\end{eqnarray}

For $i=k=0$, \eqref{eq:2.13} reduces to \eqref{eq:2.6}; for $j=k=0$, \eqref{eq:2.14} becomes \eqref{eq:2.7}.

\begin{eqnarray}
&&Hl\left ( \frac{1}{2}, (2k-1)(k+n);2(k+n),2k-1;2k,2k;x \right )\nonumber\\
&=&2^{2k-1}{n+k-1 \choose k-1}^{-1}{n-1 \choose k-1}^{-1}\sum _{i=0}^{n-k}{2n-2i-2 \choose 2k-2}{n-1 \choose i}r_{n,k+i}(1-2x)^{1-2n+2i}\nonumber \\
&=&(1-2x)^{1-2n}\sum _{j=0}^{n-k}4^j {n-k \choose j}\frac{(k+1/2)_j}{(2k)_j}(x^2-x)^j, \quad 1\leq k \leq n.\label{eq:2.15}
\end{eqnarray}

From \eqref{eq:2.15} we derive
\begin{eqnarray}
&&\sum _{j=i}^{n-k}(-1)^{j-i}{n-k \choose j}\frac{(k+1/2)_j}{(2k)_j}{j\choose i}\label{eq:2.16}\\
&=&2^{2k-1}{n+k-1 \choose k-1}^{-1}{n-1 \choose k-1}^{-1}{2n-2i-2 \choose 2k-2}{n-1\choose i}r_{n,k+i},\quad 0\leq i \leq n-k,\nonumber
\end{eqnarray}
\begin{eqnarray}
&&\sum _{i=j}^{n-k}{2n-2i-2 \choose 2k-2} {n-1 \choose i}{i \choose j}r_{n,k+i}\label{eq:2.17}\\
&=&2^{1-2k}{n+k-1 \choose k-1}{n-1 \choose k-1}{n-k \choose j}\frac{(k+1/2)_j}{(2k)_j},\quad 0\leq j \leq n-k.\nonumber
\end{eqnarray}

For $i=0$, $k=1$ and $n$ replaced by $n+1$, from \eqref{eq:2.16} we obtain
\begin{equation}
\sum _{j=0}^n \left (-\frac{1}{4}\right )^j{n \choose j}{2j+1 \choose j} = \frac{1}{(n+1)4^n}{2n \choose n}. \label{eq:2.18}
\end{equation}

With $j=0$, $k=1$, and $n$ replaced by $n+1$, \eqref{eq:2.17} yields
\begin{equation}
\sum _{i=0}^n (i+1){2i+2 \choose i+1}{2n-2i \choose n-i} = \frac{n+1}{2}4^{n+1}. \label{eq:2.19}
\end{equation}

It is a pleasant calculation to prove \eqref{eq:2.18} and \eqref{eq:2.19} directly.

\subsection*{}
Using \eqref{eq:2.8} and \eqref{eq:2.11} we get
\begin{eqnarray}
&&Hl\left ( \frac{1}{2}, (2k+1)(k+n+1);2(k+n+1),2k+1;2k+1,2k+1;x \right )\nonumber\\
&=&4^{k} {n+k \choose n}^{-1}{n \choose k}^{-1}\sum _{j=0}^{n-k}{2k+2j \choose 2j}{n \choose k+j}r_{nj}(1-2x)^{-2k-1-2j}\nonumber\\
&=&(1-2x)^{-2n-1}\sum_{j=0}^{n-k} 4^j {n-k \choose j}\frac{(k+1/2)_j}{(2k+1)_j}(x^2-x)^j, \quad 0\leq k \leq n.\label{eq:2.20}
\end{eqnarray}

Taking into account that $r_{n,n-j}=r_{nj}$, from \eqref{eq:2.20} we derive \eqref{eq:2.13} and \eqref{eq:2.14}.

Moreover,
\begin{eqnarray}
Hl\left ( \frac{1}{2}, (2k+1)(k-n);2(k-n),2k+1;2k+2,2k+2;x \right )\nonumber\\
=4^{k}\frac{2k+1}{n+1} {n+k+1 \choose n}^{-1}{n \choose k}^{-1}\sum _{j=0}^{n-k}{2k+2j+2 \choose 2j}{n+1 \choose k+j+1}r_{nj}(1-2x)^{2n-2k-2j}\nonumber\\
=\sum_{j=0}^{n-k}4^j {n-k \choose j}\frac{(k+1/2)_j}{(2k+2)_j}(x^2-x)^j, \quad 0\leq k \leq n.\label{eq:2.21}
\end{eqnarray}

From \eqref{eq:2.21} we derive
\begin{eqnarray}
&&\sum_{j=i}^{n-k}(-1)^{j-i} {n-k \choose j}\frac{(k+1/2)_j}{(2k+2)_j}{j \choose i}, \quad 0\leq i \leq n-k\label{eq:2.22}\\
&=&4^{k}\frac{2k+1}{n+1} {n+k+1 \choose n}^{-1}{n \choose k}^{-1}{2n-2i+2 \choose 2k+2}{n+1 \choose i}r_{n,k+i},\nonumber
\end{eqnarray}
\begin{eqnarray}
&&\sum_{i=j}^{n-k}{2n-2i+2 \choose 2k+2}{n+1 \choose i}{i \choose j}r_{n,k+i}, \quad 0\leq j \leq n-k\label{eq:2.23}\\
&=&4^{-k}\frac{n+1}{2k+1} {n+k+1 \choose n}{n \choose k}{n-k \choose j}\frac{(k+1/2)_j}{(2k+2)_j}.\nonumber
\end{eqnarray}

For $i=k=0$, \eqref{eq:2.22} becomes
\begin{equation}
\sum _{j=0}^n \left (-\frac{1}{4}\right )^j{n+1 \choose j+1}{2j \choose j} = \frac{2n+1}{4^n}{2n \choose n}. \label{eq:2.24}
\end{equation}

Let us recall the formula (7.6) in \cite{9}:
\begin{equation}
\sum _{j=0}^n \left (-\frac{1}{4}\right )^j{n \choose j}{2j \choose j}{j+h \choose h}^{-1} = \frac{1}{4^n}{2n+2h \choose n+h}{2h \choose h}^{-1}. \label{eq:2.25}
\end{equation}

For $h=1$, \eqref{eq:2.25} reduces to \eqref{eq:2.24}.

For $j=k=0$, \eqref{eq:2.23} becomes
\begin{equation}
\sum _{i=0}^n (2n-2i+1){2i \choose i}{2n-2i \choose n-i}=(n+1)4^n, \label{eq:2.26}
\end{equation}
which can be proved also directly.
\vspace{2cc}

\vspace{1.5cc}
\begin{center}
{\bf 3. CONFLUENT HEUN FUNCTIONS}
\end{center}
The hypergeometric function $v(t) = {_1F_1}(\alpha;\gamma;t)$ satisfies (see~\cite[p. 336]{11}, \cite[13.2.1]{13}) $v(0)=1$ and
\begin{equation}
tv''(t)+(\gamma -t)v'(t)-\alpha v(t)=0.\label{eq:3.1}
\end{equation}

Moreover (see~\cite[p. 338, 5.6]{11}, \cite[13.3.15]{13}),
\begin{equation}
{_1F_1}(\alpha +1;\gamma +1;t) = \frac{\gamma}{\alpha}\frac{d}{dt}{_1F_1}(\alpha;\gamma;t).
\label{eq:3.2}
\end{equation}

With the above notation we have:
\begin{theorem}
For $\alpha p\neq 0$, the confluent Heun function $HC(p,\gamma,0,\alpha,4p\alpha;x)$ satisfies
\begin{equation}
HC(p,\gamma,0,\alpha,4p\alpha;x)={_1F_1}(\alpha;\gamma;-4px),\label{eq:3.3}
\end{equation}
\begin{equation}
HC(p,\gamma +1,0,\alpha +1,4p(\alpha +1);x)=-\frac{\gamma}{4p\alpha}\frac{d}{dx}HC(p,\gamma,0,\alpha,4p\alpha;x),\label{eq:3.4}
\end{equation}
\begin{equation}
HC(p,\gamma +j,0,\alpha +j,4p(\alpha +j);x)=\frac{(-1)^j(\gamma)_j}{(4p)^j(\alpha)_j}\frac{d^j}{dx^j}HC(p,\gamma,0,\alpha,4p\alpha;x),\label{eq:3.5}
\end{equation}
for all integers $j\geq 0$ with $(\alpha)_j\neq 0$.
\end{theorem}

\begin{proof}
According to~\eqref{eq:1.2}, the function $u(x)=HC(p,\gamma,0,\alpha,4p\alpha;x)$ satisfies $u(0)=1$ and
\begin{equation}
xu''(x)+(4px+\gamma)u'(x)+4p\alpha u(x)=0.\label{eq:3.6}
\end{equation}

From~\eqref{eq:3.1} and \eqref{eq:3.6} it is easy to deduce that $u(x)=v(-4px)$, and this entails \eqref{eq:3.3}.

Now \eqref{eq:3.4} is a consequence of \eqref{eq:3.3} and \eqref{eq:3.2}; \eqref{eq:3.5} can be proved by induction with respect to $j$. Let us remark that \eqref{eq:3.4} coincides with (30) in \cite{6}.
\end{proof}

\begin{corollary}
Let $K_n(x) := HC \left ( n,1,0,\frac{1}{2},2n;x \right )$ be the function given by \eqref{eq:1.5}. Then
\begin{equation}
K_n(x)={_1F_1}\left (\frac{1}{2};1;-4nx \right ),\label{eq:3.7}
\end{equation}
\begin{equation}
K_n(x)=\frac{1}{\pi}\int _{-1}^1 e^{-2nx(1+t)}\frac{dt}{\sqrt{1-t^2}}.\label{eq:3.8}
\end{equation}
\end{corollary}

\begin{proof}
\eqref{eq:3.7} follows from \eqref{eq:3.3} with $\alpha = 1/2$, $\gamma=1$ and $p=n$. By using \eqref{eq:3.7} and \cite[p. 338, 5.9]{11}, \cite[13.4(i)]{13} we get \eqref{eq:3.8}. Let us remark that \eqref{eq:3.8} coincides with (69) in \cite{5}.
\end{proof}

Using \eqref{eq:3.5} with $p=n$, $\gamma =1$, $\alpha = 1/2$ we get
\begin{equation}
HC \left ( n,j+1,0,j+\frac{1}{2},2n(2j+1);x \right ) = \frac{(-1)^j}{n^j}{2j \choose j}^{-1}K_n^{(j)}(x), \quad j\geq 0.\label{eq:3.9}
\end{equation}

From \eqref{eq:3.9} and \cite[(34)]{6} we obtain
\begin{equation}
K_n^{(j)}(0) = (-n)^j {2j \choose j},\label{eq:3.10}
\end{equation}
which is (35) in \cite{6}.

On the other hand, \eqref{eq:3.8} implies (with $t=\sin \varphi$)
\begin{eqnarray}
K_n^{(j)}(0) &=&\frac{(-2n)^j}{\pi}\sum _{k=0}^j {j \choose k} \int _{-\pi/2}^{\pi/2}\sin ^k \varphi d\varphi \nonumber\\
&=&(-2n)^j\sum_{i=0}^{[j/2]}{j\choose 2i}{2i \choose i}4^{-i}.\nonumber
\end{eqnarray}

Combined with \eqref{eq:3.10}, this produces
\begin{equation*}
\sum_{i=0}^{[j/2]}{j\choose 2i}{2i \choose i}4^{-i}=2^{-j}{2j \choose j},
\end{equation*}
which is (3.99) in \cite{9}.

Finally, we give closed forms for some families of confluent Heun functions.

\begin{theorem}\label{thm:3.3}
\begin{enumerate}[(i)]
\item{}For $0\leq j \leq n$ we have
\begin{eqnarray}
&&HC\left (p,j+\frac{1}{2},0,j-n,4p(j-n);x \right )\nonumber \\ &=& \frac{(2j)!}{j!}\sum _{k=0}^{n-j}{n-j\choose k}\frac{(n-k)!}{(2n-2k)!}(16px)^{n-j-k}.\label{eq:3.11}
\end{eqnarray}

\item{}More generally, for $0\leq j\leq n$ and $\lambda >-1$,
\begin{eqnarray}
&&HC\left (p,j+1+\lambda,0,j-n,4p(j-n);x \right )\label{eq:3.12}\\ &=& \frac{(\lambda +1)_j \Gamma (\lambda +1)}{\Gamma (n+\lambda +1)}\sum _{k=0}^{n-j}(\lambda +n+1-k)_k{n-j\choose k}(4px)^{n-j-k}.\nonumber
\end{eqnarray}
\end{enumerate}
\end{theorem}

\begin{proof}
By using the relation between the function ${_1F_1}$ and the Hermite polynomials (see~\cite[p. 340, 5.16]{11}, \cite[p. 235, (4.51)]{11}, \cite[13.6.16]{13}) we have
\begin{equation}
{_1F_1}\left ( -n;\frac{1}{2};x\right ) = n! \sum _{k=0}^n \frac{1}{k!(2n-2k)!}(-4x)^{n-k}.\label{eq:3.13}
\end{equation}

From \eqref{eq:3.3} and \eqref{eq:3.13} it follows that
\begin{equation}
HC \left (p,\frac{1}{2},0,-n,-4pn;x \right ) = n! \sum _{k=0}^n \frac{(16px)^{n-k}}{k!(2n-2k)!}.\label{eq:3.14}
\end{equation}

Now \eqref{eq:3.11} is a consequence of \eqref{eq:3.14} and \eqref{eq:3.5}.

In order to prove \eqref{eq:3.12}, we need the relation between ${_1F_1}$ and the Laguerre polynomials (see \cite[p. 340, 5.14]{11}, \cite[13.6.19]{13}):
\begin{equation}
{_1F_1}(-n;\lambda +1;x) = \frac{n!\Gamma (\lambda +1)}{\Gamma (n+\lambda+1)}L_n^\lambda (x), \quad \lambda >-1,\label{eq:3.15}
\end{equation}
where (see \cite[p. 245, (4.61)]{11}, \cite[18.5.12]{13})
\begin{equation}
L_n^\lambda (x) = \sum _{k=0}^n (-1)^k \frac{(\lambda + k+1)_{n-k}}{k!(n-k)!}x^k.\label{eq:3.16}
\end{equation}

From \eqref{eq:3.3}, \eqref{eq:3.15} and \eqref{eq:3.16} we get
\begin{equation}
HC \left (p,\lambda+1,0,-n,-4pn;x \right ) = \frac{n!\Gamma (\lambda +1)}{\Gamma (n+\lambda+1)}L_n^\lambda (-4px).\label{eq:3.17}
\end{equation}

Combined with \eqref{eq:3.5}, \eqref{eq:3.17} produces\eqref{eq:3.12}, and this concludes the proof.
\end{proof}
\vspace{2cc}

\vspace{1.5cc}
\begin{center}
{\bf 4. OTHER COMBINATORIAL IDENTITIES}
\end{center}

Let us return to \eqref{eq:2.4}. Since
\begin{equation}
\frac{(m+1/2)_j}{(m+1)_j} = 4^{-j} {2m+2j \choose m+j}{2m \choose m}^{-1}, \label{eq:5.1}
\end{equation}
it becomes
\begin{eqnarray}
\sum_{j=i}^{n-m} (-1)^{j-i}4^{-j}{n-m \choose j}{2m+2j \choose m+j}{j \choose i} \nonumber\\
=4^{m-n}{n \choose m}^{-1} {m+i \choose m}{2m+2i \choose m+i} {2n-2m-2i \choose n-m-i}.\nonumber
\end{eqnarray}

Set $i+m=r$, $j=r-m+k$, and replace $n$ by $n+r$; we get
\begin{eqnarray}
\sum_{k=0}^{n} \left (-\frac{1}{4} \right )^k {n+r-m \choose n-k}{2r+2k \choose r+k}{r-m+k \choose k} \label{eq:5.2}\\
=4^{-n}{n+r \choose m}^{-1}{r \choose m}{2r\choose r}{2n\choose n}, \quad n\geq 0, r\geq m \geq 0.\nonumber
\end{eqnarray}

Here are some particular cases of \eqref{eq:5.2}.
\begin{equation*}
r=m=n:\quad \sum _{k=0}^n\left (-\frac{1}{4} \right )^k {n\choose k}{2n+2k \choose n+k}=4^{-n}{2n \choose n}.
\end{equation*}
\begin{equation*}
r=m:\quad \sum _{k=0}^n\left (-\frac{1}{4} \right )^k {n\choose k}{2r+2k \choose r+k}=4^{-n}{n+r\choose r}^{-1}{2r\choose r}{2n\choose n}.
\end{equation*}
\begin{equation*}
m=0:\quad \sum _{k=0}^n\left (-\frac{1}{4} \right )^k {n+r\choose n-k}{2r+2k \choose r+k}{r+k\choose k}=4^{-n}{2r\choose r}{2n\choose n}.
\end{equation*}
\begin{equation*}
r=n:\quad \sum _{k=0}^n\left (-\frac{1}{4} \right )^k {2n-m\choose n-k}{2n+2k \choose n+k}{n-m+k \choose k}=4^{-n}{2n\choose m}^{-1}{n\choose m}{2n\choose n}^2.
\end{equation*}
\begin{equation*}
m=n:\quad \sum _{k=0}^n\left (-\frac{1}{4} \right )^k {r\choose n-k}{2r+2k \choose r+k}{r-n+k\choose k}=4^{-n}{n+r\choose r}^{-1}{r\choose n}{2r\choose r}{2n\choose n}.
\end{equation*}

Now let us return to \eqref{eq:2.5}; use \eqref{eq:5.1}, set $j+m=r$, $i=r-m+k$, and replace $n$ by $n+r$. We get
\begin{eqnarray}
\sum_{k=0}^{n} {r+k\choose m}{r+k-m \choose k}{2r+2k\choose r+k}{2n-2k\choose n-k}\label{eq:5.3}\\
=4^{n}{n+r \choose m}{2r\choose r}{n+r-m\choose n}, \quad n\geq 0, r\geq m \geq 0.\nonumber
\end{eqnarray}

For $r=m=n$, \eqref{eq:5.3} reduces to
\begin{equation*}
\sum _{k=0}^n {n+k\choose n}{2n+2k \choose n+k}{2n-2k\choose n-k}=4^{n}{2n \choose n}^2.
\end{equation*}

Clearly, there are many other particular cases of \eqref{eq:5.3}.

Several other particular combinatorial identities can be obtained starting with other general formulas from the preceding sections, but we omit the details.

\vspace{2cc}
\subsection*{Acknowledgements}

GM is partially supported by a grant of the Romanian Ministry of National Education
and Scientific Research, RDI Programme for Space Technology and Advanced Research -
STAR, project number 513, 118/14.11.2016.

\vspace{1cc}

 %Fill author(s) affiliation(s), address(es) and emails here:

{\small
\noindent

}
\vspace{2cc}
\begin{thebibliography}{99}

\bibitem{1}I. Ra\c{s}a, \textit{Entropies and the derivatives of some Heun functions}, arXiv: 1502.05570 (2015)
\bibitem{2}R.S. Maier, \textit{On reducing the Heun equation to the hypergeometric equation}, J. Differential Equations {\bf 213} (2005), 171-203.
\bibitem{3}A. Ishkhanyan, K.A. Suominen, \textit{New solutions of Heun's general equation}, J. Phys. A: Math. Gen. {\bf 36} (2003), L81-L85.
\bibitem{4}R.S. Maier, \textit{The 192 solutions of the Heun equation}, Math. Comp. {\bf 76} (2007), 811-843.
\bibitem{5}I. Ra\c{s}a, \textit{Entropies and Heun functions associated with positive linear operators}, Appl. Math Comput. {\bf 268} (2015), 422-431.
\bibitem{6}A. B\u{a}rar, G. Mocanu, I. Ra\c{s}a, \textit{Heun functions related to entropies} (2017)
\bibitem{7}A. B\u{a}rar, G. Mocanu, I. Ra\c{s}a, \textit{Bounds for some entropies and special functions} Carpathian Journal of Mathematics {\bf 1} (2018).
\bibitem{8}I. Ra\c{s}a, \textit{R\'{e}nyi entropy and Tsallis entropy associated with positive linear operators}, arXiv:1412.4971v1 [math.CA].
\bibitem{9}H.W. Gould, \textit{Combinatorial identities}, Morgantown, W.Va. (1972).
\bibitem{10}V.A. Shahnazaryan, T.A. Ishkhanyan, T.A. Shahverdyan, A.M. Ishkhanyan, \textit{New relations for the derivative of the confluent Heun function}, Armenian J. Phys. {\bf 5} (2012), 146-156.
\bibitem{11}Gh. Mocic\u{a}, \textit{Probleme de func\c{t}ii speciale}, Edit. Didactic\u{a} \c{s}i Pedagogic\u{a}, Bucure\c{s}ti (1998).
\bibitem{12} Ronveaux A., editor. \textit{Heun's Differential Equations}. London: Oxford University Press (1995).
\bibitem{13} NIST Digital library of Mathematical Functions, \url{http://dlmf.nist.gov}
\end{thebibliography}
\end{document}